\documentclass[12pt,english]{article}
\usepackage[T1]{fontenc}
\usepackage[latin9]{inputenc}
\usepackage{geometry}
\usepackage{color}
\usepackage{babel}
\usepackage{verbatim}
\usepackage{mathtools}
\usepackage{amsmath}
\usepackage{amsthm}
\usepackage{amssymb}

 \usepackage[
  backend=biber,
 style=alphabetic,
 sortlocale=en_EN,
 sorting=nyt,
 natbib=true,
 url=false, 
 doi=true,
 eprint=false,
 giveninits=true,
 maxbibnames=10
 ]{biblatex}

\makeatletter
\theoremstyle{remark}
\newtheorem{rem}{\protect\remarkname}
\theoremstyle{definition}
\newtheorem{defn}{\protect\definitionname}
\theoremstyle{plain}
\newtheorem{thm}{\protect\theoremname}
\theoremstyle{plain}
\newtheorem*{thm*}{\protect\theoremname}
\theoremstyle{remark}
\newtheorem{claim}{\protect\claimname}

\makeatother

\providecommand{\claimname}{Claim}
\providecommand{\definitionname}{Definition}
\providecommand{\remarkname}{Remark}
\providecommand{\theoremname}{Theorem}

\begin{document}
\global\long\def\cZ{\mathbb{\mathcal{Z}}}

\global\long\def\bbc{\mathbb{C}}

\global\long\def\bbd{\mathbb{D}}

\global\long\def\bbe{\mathbb{E}}

\global\long\def\bbn{\mathbb{N}}

\global\long\def\bbp{\mathbb{P}}

\global\long\def\bbr{\mathbb{R}}

\global\long\def\pr#1{\mathbb{\bbp}\left[#1\right]}

\global\long\def\ex#1{\mathbb{\bbe}\left[#1\right]}

\global\long\def\var#1{\operatorname{Var}\left(#1\right)}

\global\long\def\im#1{\mathrm{Im}\left\{  #1\right\}  }
 
\global\long\def\zs#1{\cZ_{#1}}

\global\long\def\eqdef{\coloneqq}

\global\long\def\defeq{\eqqcolon}

\global\long\def\dd{\mathrm{d}}

\global\long\def\indf#1#2{\mathbf{1}_{#2}\left(#1\right)}

\title{Rigidity for zero sets of Gaussian entire functions}

\author{Avner Kiro\textsuperscript{1} and Alon Nishry\textsuperscript{2}}
\maketitle
\footnotetext[1]{School of Mathematical Sciences, Tel Aviv
	University, Tel Aviv, 69978, Israel. Email: avnerkiro@gmail.com. Supported in part by ERC Advanced Grant 692616 and ISF Grant 382/15.}
\footnotetext[2]{School of Mathematical Sciences, Tel Aviv
	University, Tel Aviv, 69978, Israel. Email: alonish@tauex.tau.ac.il. Supported in part by ERC Advanced Grant 692616 and ISF Grant 1903/18.}
\begin{abstract}
In this note we consider a certain class of Gaussian entire functions,
characterized by some asymptotic properties of their covariance kernels,
which we call admissible (as defined by Hayman). A notable example
is the Gaussian Entire Function, whose zero set is well-known
to be invariant with respect to the isometries of the complex plane.

We explore the rigidity of the zero set of Gaussian Taylor series, a phenomenon discovered not
long ago by Ghosh and Peres for the Gaussian Entire Function. In particular, we find that
for a function of infinite order of growth, and having an admissible
kernel, the zero set is ``fully rigid''. This means that if we know
the location of the zeros in the complement of any given compact set,
then the number and location of the zeros inside that set can be determined
uniquely. As far as we are aware, this is the first explicit construction in a natural class of random point processes with full rigidity.
\end{abstract}

\section{Introduction}

Zero sets of random analytic functions, and especially Gaussian ones,
have attracted the attention of researchers from various areas of
mathematics in the last two decades. Given a sequence $\left\{ a_{n}\right\} _{n\geq0}$ of non-negative
numbers, we consider the random Taylor series
\begin{equation}
f\left(z\right)=\sum_{n\geq0}\xi_{n}a_{n}z^{n},\label{eq:GAF_def}
\end{equation}
where $\left\{ \xi_{n}\right\} _{n\ge0}$ is a sequence of independent standard complex Gaussians. By the Cauchy-Hadamard formula, a necessary and sufficient condition for
$f$ to be almost surely (a.s.) an entire function is
\[
\lim_{n\to\infty}a_{n}^{1/n}=0,
\]
and in this case we call $f$ a \emph{Gaussian entire function} (not to be confused with the \emph{Gaussian Entire Function} which is defined below). We
will only consider transcendental entire functions, that is, sequences
$a_{n}$ which contain infinitely many non-zero terms. Denote by $\zs f=f^{-1}\left\{ 0\right\} $
the zero set of $f$; its properties are determined by the covariance
kernel
\[
K_{f}\left(z,w\right)=\ex{f\left(z\right)\overline{f\left(w\right)}}\defeq G(z\bar{w}),\quad\text{where}\quad G(z):=\sum_{n\geq0}a_{n}^{2}z^{n}.
\]

One model that was particularly
well studied is the Gaussian Entire Function (GEF), given by the Taylor
series
\[
F\left(z\right)=\sum_{n=0}^{\infty}\xi_{n}\frac{z^{n}}{\sqrt{n!}},\quad z\in\bbc,
\]
with $\xi_{n}$ as before. 
 It is well-known
that its zero set is invariant under the isometries of the complex
plane (see \cite[\S2.3]{ZerosBook}), and this fact has lead to quite an intensive study of its properties. Not long ago,
Ghosh and Peres \cite{GhoshPeres} established a \emph{rigidity phenomenon}
for the GEF. More precisely, if $K\subset\mathbb{C}$ is any compact
set, then the number of zeros in $K$, and their first moment (i.e. sum),
are uniquely determined if we are given the precise location of all zeros of $F$ in $\bbc\backslash K$. For more details on the GEF see the book \cite{ZerosBook}, and the ICM lecture notes \cite{nazarov2010random}.

It is of interest to construct point processes with higher order rigidity (e.g. \cite{ghosh2015rigidity}). We will say that a point process in $\bbc$ is ``fully rigid'' if its restriction to a compact
set $K$ is determined by its restriction to the complement $\bbc\setminus K$ (see Section \ref{sec:zero_set_rigidity} for a formal definition). To study the rigidity of the zero set of a Gaussian entire function we impose certain conditions on its covariance kernel. The \emph{admissibility} condition is a regularity property of the kernel, originally introduced by Hayman \cite{Hayman} to study generating functions, see Section \ref{sec:admissible_kernels} for the definition. In addition, we relate the rigidity of the zero set, to a certain growth condition on the function $G$, see the statement of Theorem \ref{thm:thm1}.
Two examples where the zero set exhibits full rigidity, include the kernel functions 
\[
G(z)=\exp\left(\exp\left(z\right)\right),\quad\text{and}\quad G(z)=\sum_{n\geq0}\frac{z^{n}}{\log^{n}\left(n+2\right)}.
\]
Another approach to full rigidity can be found in \cite{ghosh2018generalized}.

\subsection{Admissible kernel functions }\label{sec:admissible_kernels}

Given an entire function $G$, we put
\[
a\left(r\right)=r \frac{G^{\prime}\left(r\right)}{G\left(r\right)},\quad b\left(r\right)=ra^{\prime}\left(r\right).
\]
An entire function $G$ that satisfies the following conditions is called
\emph{admissible}:
\begin{enumerate}
\item $b\left(r\right)\to\infty$ as $r\to\infty$,
\item there exists a function, $\delta:\left[0,\infty\right)\to\left(0,\pi\right)$
such that the following holds:
\[
G\left(re^{i\theta}\right)=G\left(r\right)\exp\left(i\theta a\left(r\right)-\tfrac{1}{2}\theta^{2}b\left(r\right)\right)\left(1+o\left(1\right)\right)
\]
uniformly for $\left|\theta\right|\le\delta\left(r\right)$ as and
$r\to\infty$, and
\begin{equation}
\left|G\left(re^{i\theta}\right)\right|=\frac{o\left(G\left(r\right)\right)}{\sqrt{b\left(r\right)}},\label{eq:Hayman_cond_2}
\end{equation}
uniformly for $\delta\left(r\right)\le\left|\theta\right|\le\pi$
as $r\to\infty.$
\end{enumerate}
Let us briefly explain the meaning of each condition. The first condition
implies that the function $G$ grows sufficiently fast as $r\to\infty$.
More precisely,
\[
\log G\left(r\right)\ne O\left(\log^{2}r\right),\quad r\to\infty.
\]
Actually, our theorem relies on a much stronger growth assumption as detailed in  Remark~\ref{rmk:rigidity_growth_cond}. The second condition
implies that near $\theta=0$, the function $\theta\mapsto\log G\left(re^{i\theta}\right)$
is well-approximated by its Taylor polynomial of degree $2$ (in
$\theta$), and away from $\theta=0$, the function $\theta \mapsto G\left(re^{i\theta}\right)$ is negligibly
small compared to $G\left(r\right)$. The class of admissible functions
has many nice closure properties, for example if $f$ is admissible,
then so is $e^{f}$, for details and examples, see \cite{Hayman}.

\subsection{Rigidity of the zero set}\label{sec:zero_set_rigidity}

Ghosh and Peres \cite{GhoshPeres} introduced the following property
for point processes, in the context of zero sets of Gaussian entire
functions.
\begin{defn}
A random point processes $\mathcal{Z}$ taking values in $\bbc$,
is said to be \textit{rigid of level} $n\in\mathbb{N}^{+}\cup\{+\infty\}$,
if for any bounded open set $D\subset\mathbb{C}$ and any integer
$0\leq k<n$, there exists a map $S_{k}$ from the set of all locally
finite point configurations in $\mathbb{C}\setminus D$ to the complex
plane, such that 
\[
\sum_{D\cap\mathcal{Z}}z^{k}=S_{k}\left(\mathcal{Z}\cap\left(\mathbb{C}\setminus D\right)\right),\quad\text{almost surely.}
\]
In addition, if the point process is rigid of level $+\infty$ we
call it \emph{fully rigid}.
\end{defn}

For a fully rigid process, given the locations
of all the points outside a given compact set $K$, we can
recover the precise location of the points inside $K$, using our knowledge of all moments. More precisely, if the number of points in $D$ is $N = S_0$ (which is almost surely finite), then the points $\left\{z_j \right\}_{j=1}^N$ in $D$ are the roots of the polynomial
\[ 
\prod_{j=1}^N \left( z - z_j \right) = \sum_{\ell=0}^N (-1)^{N-\ell} e_{N-\ell}(z_1, \dots, z_N) z^\ell,
\]
where $\left\{e_\ell\right\}_{\ell = 0}^N$ are the elementary symmetric polynomials of the roots (with $e_0 \equiv 1$). It is well known that the polynomials $\left\{e_\ell\right\}$ can be computed in terms of the power sums $\left\{ S_k \right\}$ (see \cite[Prop. 7.7.1]{Stanley1999Enumerative2}).

In Section \ref{sec:Examples} we will construct explicitly Gaussian
entire functions whose zero sets are fully rigid point processes,
using our main result.
\begin{thm}
\label{thm:thm1}Let $f$ be a Gaussian entire function with covariance
kernel $(z,w)\mapsto G(z\overline{w})$. If $G$ is admissible and satisfies
\begin{equation}
B\eqdef\liminf_{r\to\infty}\frac{\log b\left(r\right)}{\log r}>0,\label{eq:growth_cond}
\end{equation}
then $\mathcal{Z}_{f}$ is rigid of level $\left\lceil B\right\rceil $,
the smallest integer larger than $B$. Moreover, if $\frac{\log b(r)}{\log r}\xrightarrow[r\to\infty]{}\infty,$
then  $\mathcal{Z}_{f}$ is fully rigid.\smallskip{}
\end{thm}
\begin{rem}
\label{rmk:rigidity_growth_cond} The condition (\ref{eq:growth_cond})
implies that $G$ is an entire function of \emph{lower order} $B$, that is
\[
\liminf_{r\to\infty}\frac{\log\log G\left(r\right)}{\log r}=B.
\]
\end{rem}
\begin{rem}
For the specific choice $a_{n}=\left(n!\right)^{-A}$
with $\frac1{2A}$ an \emph{integer}, which in particular implies (\ref{eq:growth_cond}) with $B = \frac1{2A}$,
Ghosh and Krishnapur \cite{ghosh2015rigidity} showed that the zero
set process is rigid of level $\frac1{2A} +1$. However, our Theorem only gives rigidity of level  $\frac1{2A}$, which seems to be the best result we can get from those general hypotheses.
\end{rem}
\begin{rem}
In our theorem, one could relax the usual definition of ``admissible'' by replacing the expression $\sqrt{b(r)}$ in condition (\ref{eq:Hayman_cond_2})  with $b(r)^{1/4}$.
\end{rem}

\subsubsection*{Acknowledgment}

We thank Misha Sodin for encouraging us to work on this topic, very helpful discussions, and suggestions regarding the presentation of the result. We thank an anonymous referee for carefully reading this paper, and for various helpful remarks, corrections, and suggestions.

\section{Proof of Theorem \ref{thm:thm1}\label{sec:rigidity}}

The proof is based on the following theorem, which is a special case of \cite[Theorem 6.1]{GhoshPeres}.
\begin{thm*}[Rigidity condition]
Let $n\in\bbn^{+}$. If for any bounded open set $D\subset\mathbb{C}$,
any $\varepsilon>0$, and every $0\le k<n$ there exists a smooth test function
$\Phi_{D,k}^{\varepsilon}$ which coincides with $z\mapsto z^k$ 
on $D$ and $\var{\int_{\mathbb{C}}\Phi_{D,k}^{\varepsilon}\,\dd\zs f}<\varepsilon$,
then $\zs f$ is rigid of level $n$.\smallskip{}
\end{thm*}

Let $B$ be as in the statement of Theorem \ref{thm:thm1},  $k \in [0,B)$ an integer, and $\varepsilon >0 $.  Also given $D$ a bounded open set,  choose $L \ge 1$ large enough so that $D$ is contained in $\left\{ |z| \le L \right\}$.  In the proof, $C$ will denote a positive numerical constant, independent of all
the other parameters, which may differ between appearances.  


\subsection{Definition of the test function $\Phi_{D,k}^{\varepsilon}$}
For $\eta = \eta(B, \varepsilon)>0$, which will be chosen later, let $\varphi_{\eta}:\left[0,\infty\right)\text{\ensuremath{\to\left[0,1\right]}}$
be a twice differentiable function which is equal to $1$ on $\left[0,1\right]$,
equal to $0$ on $\left[e^{\frac{1}{\eta}},\infty\right)$, and such
that 
\begin{equation}\label{deriv_bounds}
|\varphi_{\eta}'(r)|\leq\frac{C\eta}{r},\quad|\varphi_{\eta}''(r)|\leq\frac{C\eta}{r^{2}},\quad\forall r\geq0.
\end{equation}
We also put $\widehat{f}(z)=\frac{f(z)}{\sqrt{G(|z|^{2})}}$, and write $\varphi_{\eta,L}(r)=\varphi_{\eta}(r/L)$. Now we choose our test function
\begin{equation}\label{eq:test_func_def}
\Phi_{D,k}^{\varepsilon} (z) = \widetilde{\Phi}_{\eta,L}(z) \eqdef z^{k}\varphi_{\eta,L}(|z|).
\end{equation}

\subsection{Integral expression for the variance of $\widetilde{\Phi}_{\eta,L}$}

According to
\cite[equation 3.5.2]{ZerosBook}, 
\[
\var{\int_{\mathbb{C}}\widetilde{\Phi}_{\eta,L}\,\dd\zs f}=\frac{1}{4\pi^{2}}\iint_{\mathbb{C}^{2}}\Delta\widetilde{\Phi}_{\eta,L}(z)\Delta\widetilde{\Phi}_{\eta,L}(w)\ex{\log\left|\widehat{f}(z)\right|\log\left|\widehat{f}(w)\right|}\,\dd m(z)\,\dd m(w),
\]
where $m$ is the Lebesgue measure in the complex plane. Furthermore,
by \cite[Lemma 3.5.2]{ZerosBook},
\[
\ex{\log\left|\widehat{f}(z)\right|\log\left|\widehat{f}(w)\right|}=\frac{1}{4}\sum_{j\geq1}\frac{1}{j^{2}}\left|J\left(z,w\right)\right|^{j},
\]
where
\[
J\left(z,w\right)=\frac{G\left(z\overline{w}\right)}{\sqrt{G\left(\left|z\right|^{2}\right)}\sqrt{G\left(\left|w\right|^{2}\right)}},
\]
is the normalized covariance kernel. Since the Taylor coefficients
of $G$ are non-negative, $|G(z)|\leq G(\left|z\right|)$, and in
particular $\left|J\left(z,w\right)\right|\le1$. Thus 
\[
\sum_{j\geq1}\frac{1}{j^{2}}\left|J\left(z,w\right)\right|^{2j}\leq \left|J\left(z,w\right)\right|^{2} \sum_{j\geq1}\frac{1}{j^{2}} \leq2\left|J\left(z,w\right)\right|^{2},
\]
which in turn implies
\[
\var{\int_{\mathbb{C}}\widetilde{\Phi}_{\eta,L}\,\dd\zs f}\le C\iint_{\mathbb{C}^{2}}\left|\Delta\widetilde{\Phi}_{\eta,L}(z)\right|\left|\Delta\widetilde{\Phi}_{\eta,L}(w)\right|\left|J\left(z,w\right)\right|^{2}\,\dd m(z)\,\dd m(w).
\]
Recall that the Laplace operator in polar coordinates $z = r e^{i \theta}$ is given by $\frac{\partial^2}{\partial r^2} + \frac1r \frac{\partial}{\partial r} + \frac1{r^2} \frac{\partial^2}{\partial \theta^2}$. Thus, substituting \eqref{eq:test_func_def}, we compute
\[
\Delta\widetilde{\Phi}_{\eta,L}(z) = \Delta_{z,\overline{z}} \left(z^k \varphi_{\eta,L}(|z|) \right) = \Delta_{r,\theta} \left(r^k e^{i k \theta} \varphi_{\eta,L}(r) \right) = r^{k-2} e^{i k \theta} \left[(2k + 1) \tfrac{r}{L} \varphi^\prime_\eta(\tfrac{r}{L}) + \tfrac{r^2}{L^2} \varphi^{\prime\prime}_\eta(\tfrac{r}{L})\right].
\]
Hence, by \eqref{deriv_bounds} we obtain
\[
\left|\Delta\widetilde{\Phi}_{\eta,L}(z)\right| \le C \eta |z|^{k-2} (k+1).
\]
Moreover, observe that $\Delta \widetilde{\Phi}_{\eta,L}$ vanishes outside the annulus $\left\{ L \le |z| \le L e^{1/\eta}  \right\}$; this follows since $\varphi_{\eta,L}$ is supported therein, and $z^k$ is holomorphic. Combining these observations we obtain the following bound for the variance
\[
\var{\int_{\mathbb{C}}\widetilde{\Phi}_{\eta,L}\,\dd\zs f}\le C (k+1)^2 \eta^2 \iint_{\mathbb{C}^{2}} \chi_L(|z|) \chi_L(|w|) \left|J\left(z,w\right)\right|^{2}\,\dd m(z)\,\dd m(w),
\]
where $\chi_L(r)= r^{k-2}\indf {r/L}{[1,e^{1/\eta}]}$.

%

Using polar coordinates $z=L re^{i\theta_{1}}$ and $w=L se^{i\theta_{2}}$, and the definition of $J$, we finally obtain the following bound for the variance
\begin{equation}\label{eq:var_bound}
\var{\int_{\mathbb{C}}\widetilde{\Phi}_{\eta,L}\,\dd\zs f} \le C (k+1)^2 \eta^2 L^{2(k-1)}  \cdot I_L,
\end{equation}

where
\[
I_L \eqdef  \int_0^\infty \int_0^\infty \chi(r) \chi(s) \, \frac{A(L^2 r s)}{G(L^2 r^2) G(L^2 s^2)}  \, r s \, \dd r \, \dd s, \quad A(R)\eqdef\int_{-\pi}^{\pi}\int_{-\pi}^{\pi}\left|G^{2}\left(Re^{i\left(\theta_{1}-\theta_{2}\right)}\right)\right|\,\dd\theta_{1}\,\dd\theta_{2},
\]
and  $\chi \equiv \chi_1$.


\subsection{Preliminary claims}

Before bounding the integral $I_{L}$, we need two simple claims.
\begin{claim}
\label{clm:bnd_inner_int}We have
\[
A(R) \le C\cdot\frac{G^{2}\left(R\right)}{\sqrt{b\left(R\right)}},
\]
for $R$ sufficiently large.
\end{claim}
\begin{proof}
Making the linear change of variables $\theta_{2}=\theta_{2},\,\theta=\theta_{1}-\theta_{2}$
gives
\[
A(R)=\int_{-\pi}^{\pi}\int_{-\pi}^{\pi}\left|G^{2}\left(Re^{i\theta}\right)\right|\,\dd\theta\,\dd\theta_{2}=2\pi\int_{-\pi}^{\pi}\left|G^{2}\left(Re^{i\theta}\right)\right|\,\dd\theta.
\]
Recall that by admissibility
\[
\frac{\left|G\left(Re^{i\theta}\right)\right|^2}{G^2\left(R\right)}=\begin{cases}
\exp\left(-\theta^{2}b(R)\left(1+o(1)\right)\right), & \left|\theta\right|\le\delta\left(R\right);\\
\frac{o\left(1\right)}{b(R)}, & \text{otherwise},
\end{cases}
\]
as $R\to\infty$, where the $o(1)$ terms are uniform in $\theta$. Thus, for $R$ sufficiently large
\[
A(R)\leq 2\pi \cdot G^{2}\left(R\right)\left[\int_{|\theta|\leq\delta(R)}\exp\left(-\tfrac12 \theta^{2}b(R)\right)\,\dd\theta+\frac{1}{b(R)}\right].
\]
Since
\begin{equation*}
\int_{|\theta|\leq\delta(R)}\exp\left(-\tfrac12 \theta^{2}b(R) \right)\,\dd\theta \leq \int_{\mathbb{R}}\exp\left(-\tfrac12 \theta^{2}b(R)\right)\dd\theta
 \leq\frac{C}{\sqrt{b(R)}}
\end{equation*}
the claim follows.
\end{proof}
\begin{claim}
\label{clm:bound_for_norm_kernel}For $r\leq s$, 
\[
\frac{G^{2}(L^{2}rs)}{G(L^{2}r^{2})G\left(L^{2}s^{2}\right)}\leq\exp\left(-\log^{2}\frac{s}{r}\cdot\min_{x\in\left[L^{2}r^{2},L^{2}s^{2}\right]}b\left(x\right)\right).
\]
\end{claim}
\begin{proof}
Notice that if $h\left(t\right)$ is a $C^2$ convex function on the interval
$\left[x,y\right]$, the function
$$t \mapsto h\left(t\right) - \tfrac{1}{2}t^{2}\min_{s\in[x,y]}h^{\prime\prime}(s)$$
is also convex, and therefore,
\begin{equation}\label{eq:h_convex}
2h\left(\frac{x+y}{2}\right)\leq h(x)+h(y)-\frac{1}{4}\left(y-x\right)^{2}\min_{s\in[x,y]}h^{\prime\prime}(s).
\end{equation}
Recall that $G(e^t) = \sum_{n=0}^\infty a_n^2 e^{n t}$, with $a_n \ge 0$.  It follows from the Cauchy--Schwarz inequality that
\begin{align*}
G^2(e^t) \cdot \frac{\dd^2}{\dd t^2} \log G(e^t) & =G(e^t)\cdot\frac{\dd^2}{\dd t^2}G(e^t) -\left(\frac{\dd}{\dd t}G(e^t)\right)^2\\
& = \sum_{n=0}^\infty a_n^2 e^{n t} \cdot \sum_{n=0}^\infty n^2 a_n^2 e^{n t} - \left(\sum_{n=0}^\infty n a_n^2 e^{n t}  \right)^2  \geq 0,
\end{align*}
hence the function $t\mapsto\log G(e^{t})$ is convex. Applying \eqref{eq:h_convex} to this function,  with $x=\log\left(L^2 r^2\right)$ and $y=\log\left(L^2 s^2\right)$, we obtain 
\[
\log\left(\frac{G^{2}(L^{2}rs)}{G(L^{2}r^{2})G\left(L^{2}s^{2}\right)}\right)\leq-\log^{2}\frac{s}{r}\cdot\min_{t\in\left[\log\left(L^{2}r^{2}\right),\log\left(L^{2}s^{2}\right)\right]} \frac{\dd^2}{\dd t^2} \log G(e^t),
\]
and we remind that the second derivative of $\log G(e^t)$ is denoted by $b\left(e^{t}\right)$.
\end{proof}

\subsection{Bounding the integral $I_{L}$}

Recall that
\[
I_L = \int_0^\infty \int_0^\infty \chi(r) \chi(s) \, \frac{A(L^2 r s)}{G(L^2 r^2) G(L^2 s^2)}  \, r s \, \dd r \, \dd s,
\quad \chi(r)= r^{k-2}\indf {r}{[1,e^{1/\eta}]},
\]
and therefore by Claim \ref{clm:bnd_inner_int} we have that
\[
I_{L}\leq C \int_{0}^{\infty}\int_{0}^{\infty} \chi(r)\chi(s) \frac{G^{2}(L^{2}rs)}{G(L^{2}r^{2})G\left(L^{2}s^{2}\right)}\frac{1}{\sqrt{b\left(L^{2}rs\right)}}rs\,\dd r\,\dd s.
\]
The above integral is symmetric with respect to the variables $r$
and $s$, and therefore, if we put
\[
M_{L}\eqdef\min_{x\in\left[L^{2}r^{2},L^{2}s^{2}\right]}b\left(x\right),
\]
then, by Claim \ref{clm:bound_for_norm_kernel},
\begin{align*}
I_{L} & \leq C\int_{0}^{\infty}\int_{0}^{s} \chi(r)\chi(s) \frac{G^{2}(L^{2}rs)}{G(L^{2}r^{2})G\left(L^{2}s^{2}\right)}\frac{1}{\sqrt{b\left(L^{2}rs\right)}}rs\,\dd r\,\dd s\\
 & \le C\int_{0}^{\infty}\int_{0}^{s} \chi(r)\chi(s) \exp\left(-\log^{2}\frac{s}{r}\cdot M_{L}\right)\frac{1}{\sqrt{b\left(L^{2}rs\right)}}rs\,\dd r\,\dd s,
\end{align*}
By the definition of $\chi$, we get
\[
I_{L}\leq C \int_{1}^{\infty}\int_{1}^{s}\exp\left(-\log^{2}\frac{s}{r}\cdot M_{L}\right)\frac{1}{\sqrt{b\left(L^{2}rs\right)}}r^{k-1}s^{k-1}\,\dd r\,\dd s.
\]
Fix $k<B_{0}<B$. By the assumption in the statement of the theorem, for $x_0$ sufficiently large $b(x) \ge x^{B_0}$ for all $x \ge x_0$, thus there exist a constant $c>0$ (depending on $B_0$) such that $b\left(x\right)\ge cx^{B_{0}}$
for any $x\geq1$. Therefore,
\[
I_{L}\leq CL^{-B_{0}}\int_{1}^{\infty}\int_{1}^{s}\exp\left(-c\log^{2}\frac{s}{r}\cdot\left(L^{2}r^{2}\right)^{B_{0}}\right) r^{k-B_{0}/2-1}s^{k-B_{0}/2-1}\,\dd r\,\dd s.
\]
Making the change of variables $s=xr$, we obtain
\begin{align*}
I_{L} & \leq C L^{-B_{0}}\int_{1}^{\infty}\int_{1}^{\infty}\exp\left(-c\log^{2}x\cdot\left(Lr\right)^{2B_{0}}\right) r^{2k-B_{0}-1}x^{k-B_{0}/2-1}\,\dd r\,\dd x\\
 & \defeq C L^{-B_{0}}\left[I_{L}^{1}+I_{L}^{2}\right],
\end{align*}
where $I_L^1$ is the integral over the range $x \in [1, 2e]$ and $r \in [1, \infty)$, and $I_L^2$ is the integral over the rest.


To bound $I_{L}^{1},$ we use the elementary inequality $\log x>\tfrac{1}{3}(x-1)$,
for $1\leq x\leq2e,$ and for $L\ge1$ we get
\begin{align*}
I_{L}^{1} & \leq\int_{1}^{\infty}\int_{1}^{2e}\exp\left(-\frac{c}{9}\left(x-1\right)^{2}\cdot\left(Lr\right)^{2B_{0}}\right)r^{2k-B_{0}-1}x^{k-B_{0}/2-1}\,\dd x\,\dd r\\
 & \le (2e)^k \int_{1}^{\infty}\int_{1}^{2e}\exp\left(-\frac{c}{9}\left(x-1\right)^{2}\cdot\left(Lr\right)^{2B_{0}}\right)r^{2k-B_{0}-1}\,\dd x\,\dd r\\
 & \leq \frac{3 (2e)^k}{\sqrt{c}}\int_{1}^{\infty}\frac{1}{\sqrt{\left(Lr\right)^{2B_{0}}}}r^{2k-B_{0}-1}\,\dd r =  \frac{3 (2e)^k}{\sqrt{c}} \int_{1}^{\infty}r^{2k-2B_{0}-1}\,\dd r \cdot L^{-B_{0}}\defeq C_1(B_0, k) L^{-B_{0}}.
\end{align*}
Next we turn to the integral $I_{L}^{2}.$ Since $x>2e$, for $L\ge1$
we get using a linear change of variables
\begin{align*}
\int_{1}^{\infty}\exp\left(-c\log^{2}x\cdot\left(Lr\right)^{2B_{0}}\right) & r^{2k-B_{0}-1} \,\dd r  =\int_{1}^{\infty}\exp\left(-c\left(1+\log\frac{x}{e}\right)^{2}\cdot\left(Lr\right)^{2B_{0}}\right)r^{2k-B_{0}-1}\,\dd r\\
 & \leq\exp\left(-c\log^{2}\frac{x}{e}\right)\int_{1}^{\infty}\exp\left(-c\left(Lr\right)^{2B_{0}}\right)r^{2k-B_{0}-1}\,\dd r\\
 & = \exp\left(-c\log^{2}\frac{x}{e}\right)\int_{L}^{\infty}e^{-cy^{2B_{0}}}\left(\frac{y}{L}\right)^{2k-B_{0}-1}\,\frac{\dd y}{L}\\
 & \le \exp\left(-c\log^{2}\frac{x}{e}\right) L^{B_0 - 2k} \int_{1}^{\infty}e^{-cy^{2B_{0}}}y^{2k-B_{0}-1} \, \dd y \\
 & \defeq \exp\left(-c\log^{2}\frac{x}{e}\right) L^{B_0 - 2k}  \cdot C_2(B_0, k).
\end{align*}
Thus we obtain
\[
I_{L}^{2}\leq  C_2(B_0, k) L^{B_{0}-2k}\int_{2e}^{\infty}\exp\left(-c\log^{2}\frac{x}{e}\right) x^{k-B_{0}/2-1}dx \defeq C_3(B_0, k) L^{B_{0}-2k}.
\]

Since $B_0 > k$, we have established that, 
\[
I_{L}\leq C L^{-B_{0}}\left[I_{L}^{1}+I_{L}^{2}\right]\le C L^{-B_{0}}\left[C_1(B_0, k) L^{-B_{0}}+C_3(B_0, k) L^{B_{0}-2k}\right] \le C_4(B_0, k) L^{-2k},
\]
for $L\ge1$.

\subsection{Finishing the proof}
Recalling \eqref{eq:var_bound}, we conclude that
\[
\var{\int_{\mathbb{C}}\widetilde{\Phi}_{\eta,L}\,\dd\zs f} \le C (k+1)^2 \eta^2 L^{2(k-1)}  \cdot I_L \le C_5(B_0, k) \eta^2 L^{-2}.
\]
Choosing $\eta$ sufficiently small, so that $C_5(B_0, k) \eta^2<\varepsilon$
and appealing to the rigidity condition of Ghosh and Peres, we find
that $\mathcal{Z}_{f}$ is rigid of level $\left\lceil B\right\rceil $.
This completes the proof of Theorem \ref{thm:thm1}.\hfill{}$\square$

\section{Examples\label{sec:Examples}}

Here are some explicit examples for Theorem \ref{thm:thm1}.

\subsection{The Mittag-Leffler function}

We consider the function $f$ whose associated kernel function $G$
is the Mittag-Leffler function 
\[
G(z)=G_{\alpha}(z)=\sum_{n\geq0}\frac{z^{n}}{\Gamma\left(1+\alpha^{-1}n\right)},
\]
where $\alpha\in\left(0,\infty\right)$ is a fixed parameter. Notice
that $G_{1}(z)=e^{z}$ and $G_{\tfrac{1}{2}}\left(z\right)=\cosh\sqrt{z}$.
The asymptotic behavior of $G$ is well-known (see for example \cite[Section 3.5.3]{godberg2008value}).
In particular, as $|z|\to\infty$
\[
G(z)=\begin{cases}
\alpha e^{z^{\alpha}}+O\left(|z|^{-1}\right), & \left|\arg z\right|\leq\frac{\pi}{2\alpha};\\
O\left(|z|^{-1}\right), & \text{otherwise},
\end{cases}
\]

\begin{rem}
Notice that for $\alpha\in\left(0,\tfrac{1}{2}\right]$ the complement
of $\left\{ \left|\arg z\right|\leq\frac{\pi}{2\alpha}\right\} $
is empty.
\end{rem}
From this asymptotic description, one easily verifies the admissibility
assumptions are fulfilled, with
\[
a\left(r\right)\sim\alpha r^{\alpha},\quad b\left(r\right)\sim\alpha^{2}r^{\alpha},\quad r\to\infty.
\]
Thus, by Theorem \ref{thm:thm1}, $\mathcal{Z}_{f}$ is rigid of level
$\left\lceil \alpha\right\rceil $. However, for $\alpha$ an integer,
using a similar proof to \cite{ghosh2015rigidity} and the above asymptotic
formula one can check that $\mathcal{Z}_{f}$ is rigid of level $\alpha+1$.

\subsection{The double exponent and the Lindelöf functions}

Here we consider the function $f$ associated with kernel function
$G(z)=e^{e^{z}}.$ Since $e^{z}$ is admissible, so is $G$. In addition,
$G$ has an infinite order of growth, with
\[
a\left(r\right)=re^{r},\quad b\left(r\right)=r\left(r+1\right)e^{r}.
\]
By Theorem \ref{thm:thm1}, $\mathcal{Z}_{f}$ is fully rigid. 

For $\alpha>0$, we consider the function $f$ whose associated kernel
function $G$ is given by 
\[
G(z)=G_{\alpha}\left(z\right)=\sum_{n\geq0}\frac{z^{n}}{\log^{\alpha n}\left(n+e\right)}.
\]
The function $G$ has an infinite order of growth, and it follows,
for example, from \cite[example 1.4.1]{KIROSODIN}, that $G$ is admissible
with
\[
a\left(r\right)\sim\exp\left(r^{1/\alpha}-1\right),\quad b\left(r\right)\sim\exp\left(r^{1/\alpha}-\frac{\log r}{\alpha}-1-\log\alpha\right),\quad r\to\infty.
\]
Again by Theorem \ref{thm:thm1}, $\mathcal{Z}_{f}$ is fully rigid. 


\printbibliography[]

\end{document}